\documentclass[draft]{amsart}
\usepackage{graphicx,amssymb}
\newtheorem{lemma}{Lemma}
\newtheorem{proposition}{Proposition}

\newtheorem{corollary}{Corollary}

\theoremstyle{definition}
\newtheorem{remark}{Remark}
\newtheorem{conjecture}{Conjecture}
\newtheorem{example}{Example}
\newtheorem{definition}{Definition}

\newcommand{\myK}{\mathbf{\Delta Gpd}}

\newcommand{\myR}{\mathbf{Ring}}

\newcommand{\myGset}[1]{\mathbb{S}_{#1}\mathbf{Set}}
\title{$\Delta$-groupoids in knot theory}
\author{R. M.  Kashaev}
\thanks{The work is supported in part by the Swiss National Science
  Foundation.}
\address{Universit\'e de Gen\`eve,
Section de math\'ematiques,
2-4, rue du Li\`evre,
CP 64,
1211 Gen\`eve 4, Suisse
}
\date{August 9, 2009}
\email{Rinat.Kashaev@unige.ch}
\begin{document}

\begin{abstract}
A $\Delta$-groupoid is an algebraic structure which axiomitizes the combinatorics of a truncated tetrahedron. It is shown that there are relations of $\Delta$-groupoids to rings, group pairs, and (ideal) triangulations of three-manifolds. In particular, one can associate a $\Delta$-groupoid to ideal triangulations of knot complements. It is also possible to define a homology theory of $\Delta$-groupoids. The constructions are illustrated by examples coming from knot theory.
\end{abstract}
\maketitle
\section{Introduction}\label{mysec:1}
Let $G$ be a groupoid and $H$ its subset. We say that a pair of elements $(x,y)\in H\times H$ is $H$-\emph{composable} if it is composable in $G$ and $xy\in H$.
\begin{definition}\label{def:delta}
A \emph{  $\Delta$-groupoid} is a groupoid $G$, a generating subset $H\subset G$, and an involution $j\colon H\to H$, such that
\begin{itemize}
\item[(i)] $i(H)=H$, where $i(x)=x^{-1}$;
\item[(ii)] the involutions $i$ and $j$ generate an action of the symmetric group $\mathbb{S}_3$ on the set $H$, i.e. the following equation is satisfied: $iji=jij$;
\item[(iii)] if $(x,y)\in H^2$ is a composable pair then $(k(x),j(y))$ is also a composable pair, where $k=iji$;
\item[(iv)] if $(x,y)\in H^2$ is $H$-composable then $(k(x),j(y))$ is also $H$-composable, and
the following identity is satisfied:
\begin{equation}\label{myeq:1}
k(xy)ik(y)=k(k(x)j(y)).
\end{equation}
\end{itemize}
A \emph{$\Delta$-group} is a $\Delta$-groupoid  with one object (identity element).
\end{definition}
A morphism  between two $\Delta$-groupoids is a groupoid morphism $f\colon G\to G'$ such that $f(H)\subset H'$ and $j'f=fj$. In this way, one comes to the category $\myK$ of   $\Delta$-groupoids.
\begin{remark}
Equation~\eqref{myeq:1} can be replaced by an equivalent equation of the form
\[
ij(x)j(xy)=j(k(x)j(y)).
\]
\end{remark}
\begin{remark}\label{rm:inv}
In any $\Delta$-groupoid $G$ there is a canonical involution $A\mapsto A^*$ acting on the set of objects (or the identities) of $G$. It can be defined as follows. Let $\mathrm{dom},\mathrm{cod}\colon G\to \mathrm{Ob}G$ be the domain (source) and the codomain (target) maps, respectively. As $H$ is a generating set for $G$, for any $A\in\mathrm{Ob}G$ there exists $x\in H$ such that $A=\mathrm{dom}(x)$. We define $A^*=\mathrm{dom}(j(x))$. This definition is independent of the choice of $x$. Indeed, let $y\in H$ be any other element satisfying the same condition. Then, the pair $(i(y),x)$  is composable, and, therefore, so is $(ki(y),j(x))$. Thus\footnote{For compositions in groupoids we use the convention adopted for fundamental groupoids of topological spaces, i.e. $(x,y)\in G^2$ is composable iff $\mathrm{cod}(x)=\mathrm{dom}(y)$, and the product is written $xy$ rather than $y\circ x$.},
\[
\mathrm{dom}(j(y))=\mathrm{cod}(ij(y))=\mathrm{cod}(ki(y))=\mathrm{dom}(j(x)).
\]
\end{remark}
\begin{remark}
Definition~\ref{def:delta} differs from the one given in the preprint~\cite{K} in the following aspects:
\begin{enumerate}
\item the subset $H$ was not demanded to be a generating set for $G$ so that it was possible to have empty $H$ with non-empty $G$;
\item the condition (iii), which is essential in Remark~\ref{rm:inv}, was not imposed;
\item it was implicitly assumed that any $x\in H$ enters an $H$-composable pair, and under that assumption the condition (ii) was superfluous\footnote{I am grateful to D.~Bar-Natan for pointing out to this assumption during my talk at the workshop "Geometry and TQFT", Aarhus, 2007.}.
    \end{enumerate}
\end{remark}

 \begin{example}\label{myex:11}
 Let $G$ be a group. The tree groupoid $G^{2}$ is a $\Delta$-groupoid with $H=G^2$, $j(f,g)=(f^{-1},f^{-1}g)$.
  \end{example}
 \begin{example}\label{myex:0}
 Let $X$ be a set. The set
 \(
 X^{3 }
 \)
can be thought of as a disjoint sum of tree groupoids $X^2$ indexed by $X$: $X^3\simeq\sqcup_{x\in X}\{x\}\times X^2$. In particular, $\mathrm{Ob}(X^3)=X^2$ with $\mathrm{dom}(a,b,c)=(a,b)$ and $\mathrm{cod}(a,b,c)=(a,c)$. This is a $\Delta$-groupoid with $H=X^3$ and $j(a,b,c)=(b,a,c)$.
 \end{example}
\begin{example}\label{myex:01}
 We define an involution $\mathbb{Q}\cap [0,1[\ni t\mapsto t^*\in\mathbb{Q}\cap [0,1[$ by the following conditions: $0^*=0$ and if $t=p/q$ with positive mutually prime integers $p,q$, then
 $t^*=\bar p/q$, where $\bar p$ is uniquely defined by the equation $p\bar p=-1\pmod q$. We also define
 a map $\mathbb{Q}\cap [0,1[\ni t\mapsto \hat t\in\mathbb{Q}\cap [0,1]$ by the formulae $\hat 0=1$ and
 $\hat t =(p\bar p+1)/q^2$ if $t=p/q$ with positive, mutually prime integers $p,q$. Notice that in the latter case
 $\hat t=\tilde p/q$ with $\mathbb{Z}\ni\tilde p=(p\bar p+1)/q$, and $1\le\tilde p\le \min(p,\bar p)$. We also remark that $\widehat{t^*}=\hat t$.

The rational strip $X=\mathbb{Q}\times (\mathbb{Q}\cap [0,1[)$ can be given a groupoid structure as follows. Elements $(x,s)$ and $(y,t)$ are composable iff  $y\in s+\mathbb{Z}$, i.e. the fractional part of $y$ is $s$, with the product
$(x,s)(s+m,t)=(x+m,t)$, the inverse $(s+k,t)^{-1}=(t-k,s)$, and the set of units $\mathrm{Ob}X=\{(t,t)\vert\ t\in \mathbb{Q}\cap [0,1[\}$. Denote $\Gamma_X$ the underlying graph of $X$, i.e. the subset of non-identity morphisms. One can show that $X$ is a $\Delta$-groupoid with $H=\Gamma_X$ and
\[
k(x,t)=\left(\frac{t^* x-\hat t}{x-t},t^*\right).
\]
\end{example}
\begin{example}\label{myex:1}
Let $R$ be a ring. We define a $\Delta$-group $AR$ as the subgroup of the group $R^* $ of invertible elements of $R$ generated by the subset $H=(1-R^* )\cap R^*$ with $k(x)=1-x$.
\end{example}
\begin{example}\label{myex:2}
For a ring $R$, let $R\rtimes R^* $ be the semidirect product of the additive group $R$ with the multiplicative group $R^*$ with respect to the (left) action of $R^* $ on $R$ by left multiplications. Set theoretically, one has $R\rtimes R^* =R\times R^* $, the group structure being given explicitly by the product $(x,y)(u,v)=(x+yu,yv)$,  the unit element $(0,1)$, and the inversion map $(x,y)^{-1}=(-y^{-1}x,y^{-1})$. We define a $\Delta$-group $BR$ as the subgroup of $R\rtimes R^*$ generated by the subset $H=R^* \times R^*$ with $k(x,y)=(y,x)$.
\end{example}
\begin{example}
Let $(G,G_\pm,\theta)$ be a symmetrically factorized group of \cite{KR}. That means that $G$ is a group with two isomorphic subgroups $G_\pm $ conjugated to each other by an involutive element $\theta\in G$, and the restriction of the multiplication  map $m\colon G_+\times G_-\to G_+G_-\subset G$ is a set-theoretical bijection, whose inverse is called the factorization map $G_+G_-\ni g\mapsto (g_+,g_-)\in G_+\times G_-$. In this case, the subgroup of $G_+$ generated by the subset $H=G_+\cap G_-G_+\theta\cap\theta G_+G_-$ is a  $\Delta$-group with $j(x)=(\theta x)_+$.
\end{example}

The paper is organized as follows. In Section~\ref{sec:10}, we show that there is a canonical construction of a $\Delta$-groupoid starting from a group and a malnormal subgroup. In Section~\ref{sec:12}, we show that two constructions in Examples~\ref{myex:1} and \ref{myex:2} are functors which admit left adjoints functors. In Sections~\ref{sec:40} and \ref{sec:41}, we reveal a representation theoretical interpretation of the $A'$-ring of the previous section in terms of a restricted class of representations of group pairs into two-by-two upper-triangular matrices with elements in arbitrary rings. In Section~\ref{sec:pres}, in analogy with group presentations, we show that $\Delta$-groupoids can be presented starting from tetrahedral objects. In Section~\ref{sec:hom}, we define an integral homology of $\Delta$-groupoids. In the construction, actions of symmetric groups in chain groups are used in an essential way. Finally, Section~\ref{sec:concl} is devoted to a general discussion of the results of this paper, their applications in knot theory, and some open questions.

\section{$\Delta$-groupoids and pairs of groups}\label{sec:10}

Recall that a subgroup $H$ of a group $G$ is called \emph{malnormal} if the condition $gHg^{-1}\cap H\ne\{1\}$ implies that $g\in H$. In fact, for any pair of groups $H\subset G$ one can associate in a canonical way another group pair $H'\subset G'$ with malnormal $H'$.
Namely, if $N$ is the maximal normal subgroup of $G$ contained in $H$, then we define $G'=G/N$ and $H'\subset G'$ is the malnormal closure of $H/N$.

\begin{lemma}
Let a subgroup $H$ of a group  $G$ be malnormal. Then, the right
action of the group $H^3$ on the set $(G\setminus H)^2$ defined by the formula
\[
(G\setminus H)^2\times H^3\ni (g,h)\mapsto gh=(h_1^{-1}g_1h_2,h_1^{-1}g_2h_3)\in (G\setminus H)^2,
\]
\[
h=(h_1,h_2,h_3)\in H^3,\quad g=(g_1,g_2)\in (G\setminus H)^2
\]
is free.
\end{lemma}
\begin{proof}
Let $h\in H^3$ and $g\in (G\setminus H)^2$ be such that $gh=g$. On the level of components, this corresponds to two equations $h_1^{-1}g_1h_2=g_1$ and $h_1^{-1}g_2h_3=g_2$, or equivalently
$g_1h_2g_1^{-1}=h_1$ and $g_2h_3g_2^{-1}=h_1$. Together with the  malnormality of $H$ these equations imply that $h_1=h_2=h_3=1$.
\end{proof}
We provide the set of orbits $\tilde G=(G\setminus H)^2/H^3$ with a groupoid structure as follows.
Two orbits $fH^3$, $gH^3$ are composable iff $Hf_2H=Hg_1H$ with the product
\[
fH^3gH^3=(f_1,f_2)H^3(g_1,g_2)H^3=(f_1,h_0g_2)H^3,
\]
where $h_0\in H$ is the unique element such that $f_2H=h_0g_1H$. The units are $1_{HgH}=(g,g)H^3$ and the inverse of $(g_1,g_2)H^3$ is $(g_2,g_1)H^3$. Let $\Gamma(\tilde G)$ be the underlying graph which, as a set, concides with the complement of units in $\tilde G$. Clearly, it is stable under the inversion. We define the map $j\colon \Gamma( \tilde G)\to \Gamma(  \tilde G)$ by the formula:
\[
(g_1,g_2)H^3\mapsto (g_1^{-1},g_1^{-1}g_2)H^3.
\]
\begin{proposition}
The groupoid $\mathcal{G}_{G,H}=\tilde G$ is a $\Delta$-groupoid with the distinguished generating subset $H=\Gamma( \tilde G)$ and involution $j$.
\end{proposition}
\begin{proof} We verify that the map $j$ is well defined. For any $h\in H^3$ and $g\in (G\setminus H)^2$ we have
\begin{multline*}
j(ghH^3)=j((h_1^{-1}g_1h_2,h_1^{-1}g_2h_3)H^3)\\=(h_2^{-1}g_1^{-1}h_1,h_2^{-1}g_1^{-1}h_1h_1^{-1}
g_2h_3)H^3
=(h_2^{-1}g_1^{-1}h_1,h_2^{-1}g_1^{-1}g_2h_3)H^3\\=(g_1^{-1},g_1^{-1}g_2)(h_2,h_1,h_3)H^3=
(g_1^{-1},g_1^{-1}g_2)H^3
=j(gH^3).
\end{multline*}
Verification of the other properties is straightforward.
\end{proof}
\begin{example}
 For the group $G=PSL(2,\mathbb{Z})$, let the subgroup $H$ be given by the upper triangular matrices. One can show that $H$ is malnormal and the associated $\Delta$-groupoid is isomorphic to the one of Example~\ref{myex:01}.
\end{example}

\section{$\Delta$-groupoids and rings}\label{sec:12}
In this section we show that the constructions in examples~\ref{myex:1} and \ref{myex:2} come from functors admitting left adjoints.
\begin{proposition}
The mappings $A, B\colon\myR\to\myK$ are functors which admit left adjoints $A'$ and $B'$ respectively.
\end{proposition}
\begin{proof}
The case of $A$. Let $f\colon R\to S$ be a morphism of rings. Then, obviously, $f(R^* )\subset S^* $. Besides, for any $x\in R$ we have $f(k(x))=f(1-x)=f(1)-f(x)=1-f(x)=k(f(x))$, which implies that $f(k(R^* ))=k(f(R^* ))\subset k(S^* )$.  Thus, we have a well defined morphism of $\Delta$-groups $Af=f\vert_{R^* }$. If $f\colon R\to S$, $g\colon S\to T$ are two morphisms of rings, then $A(g\circ f)=g\circ f\vert_{R^* }=g\vert_{S^* }\circ f\vert_{R^* }=Ag\circ Af$. The proof of the functoriality of $B$ is similar.

We define covariant functors $A',B'\colon \myK\to\myR$ as follows. Let $G$ be a $\Delta$-groupoid. The ring $A'G$ is the quotient ring of the groupoid ring $\mathbb{Z}[G]$  (generated over $\mathbb{Z}$ by the elements $\{ u_x\vert\ x\in G\}$ with the defining relations $u_xu_y=u_{xy}$ if $x,y$ are composable) with respect to the
additional relations  $u_x+u_{k(x)}=1$ for all $x\in H$. The ring $B'G$ is generated over $\mathbb{Z}$ by the elements $\{ u_x, v_x\vert\ x\in G\}$ with the defining relations $u_xu_y=u_{xy}, v_{xy}=u_xv_y+v_x$ if $x,y$ are composable, and $u_{k(x)}=v_x, v_{k(x)}=u_x$ for all $x\in H$. If $f\in\myK(G,M)$, we define $A'f$ and $B'f$ respectively by $A'f\colon u_x\mapsto u_{f(x)}$, and $B'f\colon u_x\mapsto u_{f(x)}, v_x\mapsto v_{f(x)}$. It is straightforward to verify the functoriality properties of these constructions.

Let us now show that $A',B'$ are respective left adjoints of $A,B$.
In the case of $A', A$ we identify bijections $\varphi_{G,R}\colon\myR(A'G,R)\simeq\myK(G,AR)$ which are natural in $G$ and $R$. We define $\varphi_{G,R}(f)\colon x\mapsto f(u_x)$, whose  inverse is given by $\varphi_{G,R}^{-1}(g)\colon u_x\mapsto g(x)$. Let  $f\in\myR(A'G,R)$.
To verify the naturality in the first argument, let $h\in\myK(M,G)$. We have $A'h\in\myR(A'M,A'G)$,
$f\circ A'h\in\myR(A'M,R)$, and $\varphi_{M,R}(f\circ A'h)\in\myK(M,AR)$ with
\[
\varphi_{M,R}(f\circ A'h)\colon M\ni x\mapsto f(A'h(u_x))=f(u_{h(x)})=\varphi_{G,R}(f)(h(x)).
\]
Thus, $\varphi_{M,R}(f\circ A'h)=\varphi_{G,R}(f)\circ h$.
  To verify the naturality in the second argument,  let  $g\in\myR(R,S)$. Then,
\[
\myK(G,AS)\ni\varphi_{G,S}(g\circ f)\colon x\mapsto g(f(u_x))
=g\vert_{R^*} (\varphi_{G,R}(f)(x)).
\]
Thus, $\varphi_{G,S}(g\circ f)=Ag\circ \varphi_{G,R}(f)$. The proof in the case of $B',B$ is similar.
\end{proof}
There is a natural transformation
$\alpha\colon A\to B$ given by
\[
\myK(AR,BR)\ni\alpha_R\colon
x\mapsto (1-x,x).
\]
The following example illustrates that these adjunctions can be interesting from the viewpoint of (co)homology theory.
\begin{example}
\[
A\mathbb{Z}_2=\emptyset,\quad B\mathbb{Z}_2=(\mathbb{Z}_2,\{\{1\}\}),\quad
A'A\mathbb{Z}_2=B'B\mathbb{Z}_2=\mathbb{Z},
\]
\[
A\mathbb{Z}=\emptyset,\quad
B\mathbb{Z}=(\mathbb{Z}_2*\mathbb{Z}_2,\{\{2\},\{12\},\{{21},{121}\}\}),
\]
\[
A'A\mathbb{Z}=\mathbb{Z},\quad B'B\mathbb{Z}=\mathbb{Z}[\mathbb{Z}_2]\simeq\mathbb{Z}[t]/(t^2-1).
\]
Here, we encode the information about a $\Delta$-group into the pair $(G,H/\mathbb{Z}_2)$, where $H/\mathbb{Z}_2$ is the set of $k$-orbits in $H$, and we identify the group elements as sequences of positive integers by using the notation ${ijk\ldots}=1_i1_j1_k\ldots$, $1_i=\iota_i(1)$, with $\iota_i\colon  \mathbb{Z}_2\to \mathbb{Z}_2*\mathbb{Z}_2$ being the $i$-th canonical inclusion.
\end{example}
\section{A representation theoretical interpretation of the $A'$-ring of pairs of groups}\label{sec:40}

Let $G$ be a group with a proper malnormal subgroup $H$ (i.e. $H\ne G$), and $\mathcal{G}_{G,H}$, the associated $\Delta$-groupoid described in Section~\ref{sec:10}. The ring $A'\mathcal{G}_{G,H}$ can be thought of as being generated over $\mathbb{Z}$ by the set of invertible elements $\{ u_{x,y}\vert\ x,y\in G\setminus H\}$, which satisfy the following relations:
\begin{equation}\label{eq:21}
u_{x,y}u_{y,z}=u_{x,z},
\end{equation}
\begin{equation}\label{eq:22}
u_{xh,y}=u_{x,yh}=u_{hx,hy}=u_{x,y},\quad \forall h\in H,
\end{equation}
\begin{equation}\label{eq:23}
u_{y^{-1}x,y^{-1}}+u_{x,y}=1,\quad x\ne y,
\end{equation}
Notice that $u_{x,x}=1$ for any $x\in G\setminus H$.

Define another ring $R_{G,H}$ generated over $\mathbb{Z}$ by the set $\{s_g,v_g\vert\ g\in G\}$
subject to the following defining relations
\begin{equation}\label{eq:30}
s_1=1,\quad s_xs_y=s_{xy},\ \forall x,y\in G,
\end{equation}
the element $v_x$ is zero if $x\in H$ and invertible otherwise, and
\begin{equation}\label{eq:31}
v_{xy}=v_{y}+v_xs_y, \quad \forall x,y\in G.
\end{equation}

For any ring $R$, let
$R^*\ltimes R$ be the semidirect product of the group of invertible elements $R^*$ and the additive group $R$ with respect to the (right) action of $R^*$ on $R$ by right multiplications. As a set, the group $R^*\ltimes R$ is the product set $R^*\times R$ with the multiplication rule $(x,y)(x',y')=(xx',yx'+y')$. This construction is functorial in the sense that for any ring homomorphism $f\colon R\to S$ there corresponds a group homomorphism $\tilde f\colon R^*\ltimes R\to S^*\ltimes S$.
\begin{definition}\label{def:10}
Given a pair of groups $(G,H)$, where $H$ is a malnormal subgroup of $G$, a ring $R$ and a group homomorphism
\[
\rho\colon G\ni x\mapsto (\alpha(x),\beta(x))\in R^*\ltimes R.
\]
We say that $\rho$ is a \emph{special representation} if it satisfies the following condition: for any $x\in G$, the element $\beta(x)$ is zero for $x\in H$ and invertible otherwise.
\end{definition}
In the case of the ring $R_{G,H}$ we have the canonical special representation
\[
\sigma_{G,H}\colon G\to R_{G,H}^*\ltimes R_{G,H},\quad x\mapsto(s_x,v_x),
\]
which is universal in the sense that for any ring $R$ and any special representation $\rho\colon G\to R^*\ltimes R$ there exists a unique ring homomorphism $f\colon R_{G,H}\to R$ such that $\rho=\tilde f\circ \sigma_{G,H}$.

Let us define a map $q$ of the generating set of the ring $A'\mathcal{G}_{G,H}$ into the ring $R_{G,H}$ by the formula
\[
q(u_{x,y})=v_{x^{-1}}v_{y^{-1}}^{-1}.
\]
\begin{proposition}
The map $q$ extends to a unique ring homomorphism \(
q\colon  A'\mathcal{G}_{G,H}\to R_{G,H}
\).
\end{proposition}
\begin{proof}
The elements $q(u_{x,y})$ are manifestly invertible and satisfy the identities  $q(u_{x,z})=q(u_{x,y})q(u_{y,z})$. The consistency of the map $q$ with relations~\eqref{eq:22} is easily seen from the following properties of the elements $v_x$ (which are special cases of equation~\eqref{eq:31}):
\[
v_{hx}=v_x,\quad v_{xh}=v_xs_h,\quad \forall h\in H.
\]
The identity $q(u_{y^{-1}x,y^{-1}})+q(u_{x,y})=1$ is equivalent to
\[
v_{x^{-1}y}=v_y-v_{x^{-1}}v_{y^{-1}}^{-1}v_y
\]
which, in turn, is equivalent to the defining relation~\eqref{eq:31} after taking into account the formula
\[
s_y=-v_{y^{-1}}^{-1}v_y.
\]
The latter formula follows from the particular case of the relation~\eqref{eq:31} corresponding to $x=y^{-1}$.
\end{proof}
Let us fix an element $g\in G\setminus H$ and define a map $f_g$ of the generating set of the ring $R_{G,H}$ into
the ring $A'\mathcal{G}_{G,H}$
by the following formulae
\begin{equation}\label{eq:33}
f_g(s_x)=\left\{\begin{array}{rl}
u_{g,xg}=u_{x^{-1}g,g},&\mathrm{if}\  x\in H;\\
-u_{g,x}u_{x^{-1},g},&\mathrm{otherwise},
\end{array}\right.
\end{equation}
and
\begin{equation}\label{eq:32}
f_g(v_{x})=\left\{\begin{array}{rl}
0,&\mathrm{if}\  x\in H;\\
u_{x^{-1},g},&\mathrm{otherwise}.
\end{array}\right.
\end{equation}
\begin{proposition}
For any $g\in G\setminus H$ the map $f_g$ extends to a unique ring homomorphism \(
f_g\colon R_{G,H}\to A'\mathcal{G}_{G,H}
\).
\end{proposition}
\begin{proof}
1. Clearly, $f_g(1)=f_g(s_1)=u_{g,g}=1$ and $f_g(s_x^{-1})=f_g(s_{x^{-1}})=f_g(s_x)^{-1}$ if $x\not\in H$. For $x\in H$ we have
\[
f_g(s_x^{-1})=f_g(s_{x^{-1}})=u_{xg,g}=u_{g,xg}^{-1}=f_g(s_x)^{-1}.
\]

2. We check the identity $f_g(s_xs_{y})=f_g(s_{xy})=f_g(s_x)f_g(s_y)$ in five different cases:
\begin{description}
\item[($x,y\in H$)]
\[
f_g(s_{xy})=u_{g,xyg}=u_{x^{-1}g,yg}=u_{x^{-1}g,g}u_{g,yg}=f_g(s_x)f_g(s_y);
\]
\item[($x\in H$, $y\not\in H$)]
\[
f_g(s_{xy})=-u_{g,xy}u_{y^{-1}x^{-1},g}=-u_{x^{-1}g,y}u_{y^{-1},g}=-u_{x^{-1}g,g}u_{g,y}u_{y^{-1},g}=
f_g(s_x)f_g(s_y);
\]
\item[($x\not\in H$, $y\in H$)]
\[
f_g(s_{xy})=-u_{g,xy}u_{y^{-1}x^{-1},g}=-u_{g,x}u_{x^{-1},yg}=-u_{g,x}u_{x^{-1},g}u_{g,yg}=
f_g(s_x)f_g(s_y);
\]
\item[($x\not\in H$, $y\not\in H$, $xy\in H$)]
\[
f_g(s_{xy})=f_g(s_{xy})f_g(s_y^{-1})f_g(s_y)=f_g(s_{xy}s_y^{-1})f_g(s_y)=f_g(s_x)f_g(s_y);
\]
\item[($x\not\in H$, $y\not\in H$, $xy\not\in H$)]
\begin{multline*}
f_g(s_{xy})=-u_{g,xy}u_{y^{-1}x^{-1},g}=-u_{g,x}u_{xy,x}^{-1}u_{y^{-1}x^{-1},y^{-1}}u_{y^{-1},g}\\
=-u_{g,x}(1-u_{y,x^{-1}})^{-1}(1-u_{x^{-1},y})u_{y^{-1},g}
=u_{g,x}(u_{y,x^{-1}}-1)^{-1}(u_{y,x^{-1}}-1)u_{x^{-1},y}u_{y^{-1},g}\\
=u_{g,x}u_{x^{-1},y}u_{y^{-1},g}
=u_{g,x}u_{x^{-1}g}u_{g,y}u_{y^{-1},g}=f_g(s_x)f_g(s_y).
\end{multline*}
\end{description}
3. We check the identity $f_g(v_{xy})=f_g(v_y)+f_g(v_x)f_g(s_y)$ in five different cases:
\begin{description}
\item[($x,y\in H$)] it is true trivially;
\item[($x\in H$, $y\not\in H$)]
\[
f_g(v_{xy})=u_{y^{-1}x^{-1},g}=u_{y^{-1},g}=f_g(v_y);
\]
\item[($x\not\in H$, $y\in H$)]
\[
f_g(v_{xy})=u_{y^{-1}x^{-1},g}=u_{x^{-1},yg}=u_{x^{-1},g}u_{g,yg}=f_g(v_x)f_g(s_y);
\]
\item[($x\not\in H$, $y\not\in H$, $xy\in H$)]
\begin{multline*}
f_g(v_y)+f_g(v_x)f_g(s_y)=u_{y^{-1},g}-u_{x^{-1},g}u_{g,y}u_{y^{-1},g}\\=
(1-u_{x^{-1},y})u_{y^{-1},g}=(1-u_{x^{-1}xy,y})u_{y^{-1},g}=0;
\end{multline*}
\item[($x\not\in H$, $y\not\in H$, $xy\not\in H$)]
\begin{multline*}
f_g(v_{xy})=u_{y^{-1}x^{-1},g}=u_{y^{-1}x^{-1},y^{-1}}u_{y^{-1},g}\\
=(1-u_{x^{-1},y})u_{y^{-1},g}=
(1-u_{x^{-1},g}u_{g,y})u_{y^{-1},g}=f_g(v_y)+f_g(v_x)f_g(s_y).
\end{multline*}
\end{description}
\end{proof}
Associated with any invertible element  $t$ of the ring $R_{G,H}$ there is an endomorphism
$r_t\colon R_{G,H}\to R_{G,H}$ defined on the generating elements by the formulae
\[
r_t(s_x)=ts_xt^{-1},\quad r_t(v_x)=v_x t^{-1}.
\]
Note that, in general, $r_t$ can have a non-trivial kernel, for example, if $g\in G\setminus H$, then
$1-v_g\in \mathrm{ker}(r_{v_g})$.
\begin{proposition}
For any $g\in G\setminus H$ the following identities of ring homomorphisms
\begin{equation}\label{eq:700}
f_g\circ q=\mathrm{id}_{A'\mathcal{G}_{G,H}},\quad q\circ f_g=r_{v_{g^{-1}}}
\end{equation}
hold true.
\end{proposition}
\begin{proof}
Applying the left hand sides of the identities to be proved to the generating elements of the corresponding rings, we have
\[
f_g(q(u_{x,y}))=f_g(v_{x^{-1}}v_{y^{-1}}^{-1})=u_{x,g}u_{y,g}^{-1}=u_{x,y},\quad \forall x,y\not\in H,
\]
\[
q(f_g(s_x))=q(u_{x^{-1}g,g})=v_{g^{-1}x}v_{g^{-1}}^{-1}=v_{g^{-1}}s_xv_{g^{-1}}^{-1},\quad x\in H,
\]
\[
q(f_g(s_x))=q(-u_{g,x}u_{x^{-1},g})=-v_{g^{-1}}v_{x^{-1}}^{-1}v_xv_{g^{-1}}^{-1}
=v_{g^{-1}}s_xv_{g^{-1}}^{-1},\quad \forall x\not\in H,
\]
and
\[
q(f_g(v_x))=q(u_{x^{-1},g})=v_xv_{g^{-1}}^{-1},\quad \forall x\not\in H.
\]
\end{proof}
\begin{corollary}
The ring homomorphism $q\colon A'\mathcal{G}_{G,H}\to R_{G,H}$ is injective.
\end{corollary}
\begin{corollary}
For any $g\in G\setminus H$, the kernel of the ring homomorphism $f_g\colon R_{G,H}\to A'\mathcal{G}_{G,H}$ is generated by the element $1-v_{g^{-1}}$.
\end{corollary}
\begin{proof}
Let $t=v_{g^{-1}}$.
Equations~(\ref{eq:700}), imply that $\mathrm{ker}(f_g)=\mathrm{ker}(r_t)$ and  $r_t\circ r_t=r_t$.
The latter equation means that any $x\in \mathrm{ker}(r_t)$ has the form $y-r_t(y)$ for some $y$. The identity
\[
xy-r_t(xy)=(x-r_t(x))y+r_t(x)(y-r_t(y)),
\]
implies that that $\mathrm{ker}(r_t)$ is generated by elements $x-r_t(x)$, with $x$ running in a generating set for the ring  $R_{G,H}$. Finally, the identities
\[
s_x-r_t(s_x)=s_x-ts_xt^{-1}=(1-t)s_x-ts_xt^{-1}(1-t),\quad v_x-r_t(v_x)=-v_xt^{-1}(1-t),
\]
imply that $\mathrm{ker}(r_t)$ is generated be only one element $1-t=1-v_{g^{-1}}$.
\end{proof}
\begin{corollary}
For any $g\in G\setminus H$, the rings $A'\mathcal{G}_{G,H}$ and $R_{G,H}/(1-v_g)$ are isomorphic.
\end{corollary}
\begin{example}
Consider the group pair $(G,H)$, where
\[
G=\langle a,b\vert\ a^2=b^2=1\rangle\simeq\mathbb{Z}_2*\mathbb{Z}_2,
\]
and $H=\langle a\rangle\simeq \mathbb{Z}_2$, a malnormal subgroup of $G$. Any element of the group $G$ has the form $(ab)^ma^n$ for unique $m\in\mathbb{Z}$, $n\in\{0,1\}$. One can show that
\[
R_{G,H}\simeq\mathbb{Q}[x,x^{-1}],\quad v_{(ab)^ma^n}\mapsto (-1)^nmx,\quad s_{(ab)^ma^n}\mapsto (-1)^n,
\]  while $q(u_{(ab)^ma^n,(ab)^ka^l})\mapsto m/k$, $m,k\ne 0$. Thus, $A'\mathcal{G}_{G,H}\simeq\mathbb{Q}$.
\end{example}

\section{Weakly special representations of group pairs}\label{sec:41}
In the case of arbitrary group pairs we shall use a weakened version of special representations.
\subsection{The group pairs $(G_\rho,H_\rho)$} To each group pair homomorphism of the form
\begin{equation}\label{eq:590}
\rho\colon (G,H) \to (R^*\ltimes R,R^*),\quad G\ni g\mapsto (\alpha(g),\beta(g))\in R^*\ltimes R,
\end{equation} where $R$ is a ring,
we associate a group pair $(G_\rho,H_\rho)$, where
\[
G_\rho=G/\mathrm{ker}(\rho),\quad H_\rho=\beta^{-1}(0)/\mathrm{ker}(\rho)\subset G_\rho.
\]
The set $\beta\circ\alpha^{-1}(1)\subset R$ is an additive subgroup, and the set  $\alpha(G)\subset R^*$ is a multiplicative subgroup. These groups fit into the following short exact sequence of group homomorphisms
\begin{equation}\label{eq:600}
0\to \beta\circ\alpha^{-1}(1)\to G_\rho \to \alpha(G)\to 1,
\end{equation}
which sheds some light on the structure of $G_\rho$. In particular, there exists another short exact sequence of group homomorphisms
\[
1\to N\to G\ltimes \beta\circ\alpha^{-1}(1)\to G_\rho\to 1,
\]
 where the semidirect product is taken with respect to the right group action by group automorphisms
\[
\beta\circ\alpha^{-1}(1)\times G\ni (x,g)\mapsto x\alpha(g)\in\beta\circ\alpha^{-1}(1),
\]
and  $N$ (the image thereof) trivially intersects $\beta\circ\alpha^{-1}(1)$. We also have a group isomorphism
\[
H_\rho\simeq \alpha\circ\beta^{-1}(0)\subset \alpha(G).
\]
For a given group pair $(G,H)$ the set of homomorphisms~\eqref{eq:590} is partially ordered with respect to the relation:
\begin{equation}\label{eq:610}
\rho<\sigma\ \Leftrightarrow\ \exists\ \mathrm{exact}\
(1,1)\to (N,M)\to (G_\rho,H_\rho)\to (G_\sigma,H_\sigma)\to (1,1).
\end{equation}
\subsection{The universal ring $\hat R_{G,H}$}

Let $(G,H)$ be a group pair. The ring $\hat R_{G,H}$ is generated over $\mathbb{Z}$ by the set $\{s_g,v_g\vert\ g\in G\}$
subject to the following defining relations:
\begin{enumerate}
\item the elements $s_x$ are invertible;
\item the map
\(
\sigma_{G,H}\colon G\ni x\mapsto (s_x,v_x)\in \hat R_{G,H}^*\ltimes \hat R_{G,H}
\)
is a group homomorphism;
\item for any $x\in H$,
\(
v_{x}=0.
\)
\end{enumerate}
Notice that, according to this definition, a non-zero generating element $v_x$ is not assumed to be invertible.
This ring has the following universal property: for any group pair homomorphism~\eqref{eq:590} there exists a unique ring homomorphism $f_\rho\colon \hat R_{G,H}\to R$ such that $\rho=\tilde f_\rho\circ \sigma_{G,H}$.
The partial order~\eqref{eq:610} can alternatively be characterized by an equivalent condition:
\begin{equation}\label{eq:613}
\rho < \sigma\quad\Leftrightarrow\quad \mathrm{ker}(f_\rho)\subset\mathrm{ker}(f_\sigma).
\end{equation}
\subsection{Weakly special representations}
\begin{definition}
 Given a pair of groups $(G,H)$, where $H$ is a proper (i.e. $H\ne G$) but not necessarily malnormal subgroup of $G$, a nontrivial (i.e. $0\ne 1$) ring $R$, and a group pair homomorphism~\eqref{eq:590}.
We say that $\rho$ is a \emph{weakly special representation} if it satisfies the following conditions: (1) $\beta(G)\subset R^*\sqcup\{0\}$; (2) $\beta(G)\ne\{0\}$.
\end{definition}
Any special representation (if $H\ne G$) is also weakly special. For any weakly special representation $\rho$, the group
\(
H_{\rho}
\) is a malnormal subgroup in $G_\rho$, and the induced representation of the pair $(G_\rho,H_\rho)$ is special in the sense of Definition~\ref{def:10}.

To a group pair $(G,H)$ we associate  a set valued invariant $W(G,H)$ consisting of minimal (with respect to the partial ordering~\eqref{eq:610}) weakly special representations (considered up to equivalence). Notice that if $\rho\in W(G,H)$, then the ring $A'\mathcal{G}_{G_\rho,H_\rho}$ is non-trivial (i.e. $0\ne 1$).
Taking into account characterization~\eqref{eq:613}, there is a bijection between the set $W(G,H)$ and the set of ring homomorphisms $f_\rho$ with minimal kernel.

The following proposition will be useful for calculations.
\begin{proposition}\label{pro:100}
Given a group pair $(G,H)$ with $H\ne G$, a non-trivial ring $R$, and a weakly special representation
\[
\rho\colon (G,H)\to (R^*\ltimes R,R^*),\quad G\ni g\mapsto (\alpha(g),\beta(g))\in R^*\ltimes R.
\]
Assume that
the subring $f_\rho(\hat R_{G,H})\subset R$ is generated over $\mathbb{Z}$ by the set $\alpha(G)$.
\begin{description}
\item[(i)]
If $\beta\circ\alpha^{-1}(1)\ne\{ 0\}$ then $f_\rho(\hat R_{G,H})$ is a skew-field.
\item[(ii)]
If $1\in \beta\circ\alpha^{-1}(1)$, then $G_\rho\simeq\alpha(G)\ltimes \beta\circ\alpha^{-1}(1)$.
\end{description}
\end{proposition}
\begin{proof}
{\bf (i)}
We remark that $\alpha^{-1}(1)$ is a normal subgroup of $G$ having the following properties:
 \begin{equation}\label{eq:500}
 \beta(g^{-1}tg)=\beta(t)\alpha(g),\quad \forall g\in G,\ \forall t\in \alpha^{-1}(1).
 \end{equation}
 \begin{equation}\label{eq:510}
 \beta(t_1t_2)=\beta(t_1)+\beta(t_2),\quad \forall t_1,t_2\in\alpha^{-1}(1).
 \end{equation}
 Let $t_0\in \alpha^{-1}(1)$ be such that $x_0=\beta(t_0)$ is invertible. As any element $x\in f_\rho(\hat R_{G,H})$ can be written in the form
 \[
 x=\sum_{i=1}^n m_i\alpha(g_i),\quad m_i\in \mathbb{Z},\ g_i\in G,
 \]
we have
 \begin{multline*}
 x_0x=\sum_{i=1}^n m_ix_0\alpha(g_i)=\sum_{i=1}^n m_i\beta(t_0)\alpha(g_i)=\sum_{i=1}^n \beta(t_0^{m_i})\alpha(g_i)\\=
 \sum_{i=1}^n \beta(g_i^{-1}t_0^{m_i}g_i)=\beta(\prod_{i=1}^ng_i^{-1}t_0^{m_i}g_i)=\beta(t_x),\quad
 t_x=\prod_{i=1}^ng_i^{-1}t_0^{m_i}g_i.
 \end{multline*}
 Thus, if $x=x_0^{-1}\beta(t_x)\ne 0$, then it is  invertible.

 {\bf (ii)} Let $t_0\in\alpha^{-1}(1)$ be such that $\beta(t_0)=1$. We show that the exact sequence~\eqref{eq:600} splits.

 Let $\xi\colon \alpha(G)\to G$ be a set-theoretical section to $\alpha$ (i.e. $\alpha\circ\xi=\mathrm{id}_{\alpha(G)}$). For any $x\in\alpha(G)$ fix a (finite) decomposition
 \[
 \beta(\xi(x))=\sum_{i}m_i(x)\alpha(f_i(x)),
 \]
 where $m_i(x)\in\mathbb{Z}$, $f_i(x)\in G$.
 We define
 \[
 \sigma\colon \alpha(G)\to G_\rho,\quad x\mapsto\pi_\rho\left(\xi(x)\prod_{i}f_i(x)^{-1}t_0^{-m_i(x)}f_i(x)\right),
 \]
 where $\pi_\rho\colon G\to G_\rho$ is the canonical projection. Then, it is straightforward to see that $\sigma$
 is a group homomorphism such that $\alpha\circ\sigma=\mathrm{id}_{\alpha(G)}$.
\end{proof}

Below, we give two examples, coming from knot theory, which indicate the fact that the set $W(G,H)$ can be an interesting and relatively tractable invariant.

\begin{example}[the trefoil knot $3_1$]
Consider the pair of groups $(G,H)$, where
\[
G=\langle a,b\vert \ a^2=b^3\rangle
\] is the fundamental group of the complement of the trefoil knot, and $H=\langle ab^{-1},a^2\rangle$, the peripheral subgroup generated by the meridian $m=ab^{-1}$ and the longitude $l=a^2(ab^{-1})^{-6}$. As the element $a^2$ is central, the subgroup $H$ is not malnormal in $G$. If we take the quotient group with respect to the center
\[
G/\langle a^2\rangle\simeq \mathbb{Z}_2*\mathbb{Z}_3\simeq PSL(2,\mathbb{Z}),
\]
then the image of the subgroup $H/\langle a^2\rangle\simeq\mathbb{Z}$ is malnormal, and it is identified with the subgroup of upper triangular matrices in $PSL(2,\mathbb{Z})$. Thus, one can construct the $\Delta$-groupoid $\mathcal{G}_{\tilde G,\tilde H}$ (which is isomorphic to the one of Example~\ref{myex:01}), but the corresponding $A'$-ring happens to be trivial (i.e. $0=1$). This is a consequence of the result to be proved below: the set $W(G,H)$ consists of a single element $\rho_0$ such that, the $A'$-ring of the pair $(G_{\rho_0},H_{\rho_0})$ is isomorphic to the field $\mathbb{Q}[t]/(\Delta_{3_1}(t))$, where
\(
\Delta_{3_1}(t)=t^2-t+1
\)
is the Alexander polynomial of the trefoil knot.

The ring $\hat R_{G,H}$ admits the following finite presentation: it is generated over $\mathbb{Z}$ by four elements $s_a,s_b, v_a,v_b$, of which $s_a,s_b$ are invertible, subject to four relations
\[
s_a^2=s_b^3,\quad v_{a}=v_{b},\quad v_a(1+s_a)=v_b(1+s_b+s_b^2)=0.
\]
We consider a ring homomorphism
\[
\phi\colon \hat R_{G,H}\to \mathbb{Z}[t]/(\Delta_{3_1}(t)),\quad
s_a\mapsto -1,\quad s_b\mapsto -t^{-1},\quad v_a\mapsto 1,\quad v_b\mapsto 1.
\]
\begin{proposition}
(i) For any weakly special representation $\rho'$ there exists an equivalent representation $\rho$ such that the ring homomorphism $f_\rho$ factorizes through $\phi$, i.e. there exists a unique ring homomorphism $h_\rho\colon \mathbb{Z}[t]/(\Delta_{3_1}(t))\to R$ such that
$f_\rho=h_\rho\circ\phi$.

(ii) The kernel of the group homomorphism $\tilde\phi\circ\sigma_{G,H}$ is generated by $a^2$ and $(ab^{-1})^6$ with the quotient group pair $(\tilde G,\tilde H)$,
\[
\tilde G=\langle a,b\vert\ a^2=b^3=(ab^{-1})^6=1\rangle,\quad
\tilde H=\langle ab^{-1}\rangle,
\]
where $\tilde H$  is malnormal in $\tilde G$.

(iii) The ring $A'\mathcal{G}_{\tilde G,\tilde H}$ is isomorphic to the field $\mathbb{Q}[t]/(\Delta_{3_1}(t))$.
\end{proposition}
\begin{proof}
(i) Let $R$ be a non-trivial ring and
\[
\rho\colon G\ni g\mapsto (\alpha(g),\beta(g))\in R^*\ltimes R,
\]a weakly special representation. We have
\[
\beta(a)=\beta(ab^{-1}b)=\beta(ab^{-1})\alpha(b)+\beta(b)=\beta(b),
\]
and
\[
0=\beta(a^2)=\beta(a)(\alpha(a)+1),\quad
0=\beta(b^3)=\beta(b)(\alpha(b)^2+\alpha(b)+1).
\]
The element  $\xi=\beta(a)=\beta(b)$ is invertible, since otherwise $\xi=0$ and $\beta^{-1}(0)=G$.
Thus, $\alpha(a)=-1$ and $\alpha(b)$ is an element satisfying the equation $\Delta_{3_1}(-\alpha(b)^{-1})=0$. Replacing $\rho$ by an equivalent representation, we can assume that $\xi=1$.

The ring homomorphism $f_\rho$ is defined by the images of the generating elements
\[
s_a\mapsto -1,\quad s_b\mapsto \alpha(b),\quad v_a\mapsto 1,\quad v_b\mapsto 1,
\]
and it is easy to see that we have a factorization $f_\rho=h_\rho\circ\phi$, with a unique ring homomorphism
\[
h_\rho\colon \mathbb{Z}[t]/(\Delta_{3_1}(t))\to R,\quad t\mapsto -\alpha(b)^{-1}.
\]

(ii) It is easily verified that $a^2,(ab^{-1})^6\in\mathrm{ker}(\tilde\phi\circ\sigma_{G,H})$. We remark an isomorphism
\[
\tilde G\simeq\langle s,t_1,t_2\vert\ s^6=1,\ t_1t_2=t_2t_1,\ st_1=t_2s,\ t_1s t_2=t_2s\rangle\simeq \mathbb{Z}_6\ltimes\mathbb{Z}^2
\]
given, for example, by the formulae:
\[
s\mapsto ab^{-1},\ t_1\mapsto babab,\ t_2\mapsto bab^{-1}a,
\]
and the induced from $\tilde\phi\circ\sigma_{G,H}$ group homomorphism takes the form
\[
s\mapsto (t,0),\quad t_1\mapsto (1,1),\quad t_2\mapsto (1,t^{-1}),
\]
so that a generic element  $t_1^mt_2^ns^k$, $m,n\in\mathbb{Z}$, $k\in\mathbb{Z}_6$,  has the image
\(
(t^k,(m+nt^{-1})t^{k}).
\)
The latter is the identity element $(1,0)$ if and only if $k=0\pmod 6$ and $m=n=0$.

(iii) The pair $(G,H)$ and any weakly special representation $\rho$ satisfy the conditions of proposition~\ref{pro:100}, and thus the ring $A'\mathcal{G}_{G_\rho,H_\rho}$ is the localization of a homomorphic image of the ring $\mathbb{Z}[t]/(\Delta_{3_1}(t))$ at all non-zero elements. The ring $\mathbb{Z}[t]/(\Delta_{3_1}(t))$ itself is a commutative integral domain, and thus its minimal quotient ring corresponds to the zero ideal. In this way, we come to the field
$\mathbb{Q}[t]/(\Delta_{3_1}(t))$ which is the $A'$-ring associated to the only minimal weakly special representation $\rho_0$ with $(G_{\rho_0},H_{\rho_0})=(\tilde G,\tilde H)$.
\end{proof}
\begin{corollary}
The set $W(G,H)$ is a singleton consisting of a minimal weakly special representation $\rho_0$ such that $H\cap\mathrm{ker}(\rho_0)=\langle m^6,l\rangle$.
\end{corollary}
\begin{remark}
The group pair $(\tilde G,\tilde H)$ is the quotient of the pair $(G,H)$ with respect to the normal subgroup of $G$ generated by the center of $G$ and the longitude $l=a^2(ab^{-1})^{-6}$.
\end{remark}
\end{example}
\begin{example}[the figure-eight knot $4_1$]\label{ex:4_1} The fundamental group of the complement admits the following presentation
\[
G=\langle a_1,a_2\vert\ a_1w_a=w_aa_2\rangle,\quad w_a=a_2^{-1}a_1a_2a_1^{-1},
\]
the peripheral subgroup being given by
\[
H=\langle m=a_1,l=w_a\bar w_a\rangle\simeq\mathbb{Z}^2,\quad \bar w_a=a_1^{-1}a_2a_1a_2^{-1}.
\]
One can show that $H$ is a malnormal subgroup of $G$ (this is true for any hyperbolic knot), and that the corresponding  $A'$-ring is trivial. The latter fact will follow from our description of the set $W(G,H)$: it consists of two minimal weakly special representations $\rho_i$, $i\in\{1,2\}$, such that
\[
H\cap\mathrm{ker}(\rho_i)=\langle m^{p_i},lm^{q_i}\rangle,\quad (p_1,q_1)=(0,0),\quad (p_2,q_2)=(6,3).
\]

The ring $\hat R_{G,H}$ admits the following finite presentation: it is generated by four elements
$\{s_{a_i},v_{a_i}\vert\ i\in\{1,2\}\}$, with $s_{a_i}$ being invertible, subject to  four relations:
\[
s_{a_1}s_{w_a}=s_{w_a}s_{a_2},\quad
v_{a_1}=0,\quad
v_{\bar w_a}+v_{w_a}s_{\bar w_a}=0,\quad v_{w_a}(1-s_{a_2})=v_{a_2},
\]
where
\[
s_{w_a}=s_{a_2}^{-1}s_{a_1}s_{a_2}s_{a_1}^{-1},\quad  s_{\bar w_a}=s_{a_1}^{-1}s_{a_2}s_{a_1}s_{a_2}^{-1},
\]
\[
v_{w_a}=v_{a_2}(s_{a_1}^{-1}-s_{w_a}),\quad v_{\bar w_a}=-v_{a_2}(1-s_{a_1})s_{a_2}^{-1}.
\]
We also define a ring $S$ generated over $\mathbb{Z}$ by three elements $\{p,r,x\}$, and the following defining relations:
\begin{eqnarray}
x^2&=&x-p,\label{eq:280}\\
 px&=&x+3p+r^2-1,\label{eq:290}\\
 pr&=&r,\label{eq:300}\\
  rx+xr&=&r-r^2,\label{eq:310}\\
p^2&=&1-4p-2r^2.\label{eq:320}
\end{eqnarray}
We remark that this ring is noncommutative and finite dimensional over $\mathbb{Z}$ with $\mathrm{dim}_{\mathbb{Z}}S=6$. Indeed, it is straightforward to see that for a $\mathbb{Z}$-linear basis one can choose, for example, the set $\{1,p,r,x, rx, r^2\}$, the set $\{1,p,r^2\}$ being a $\mathbb{Z}$-basis of the center.
\begin{lemma}\label{lem:15}
In the ring $S$, let $(a)$ be the two sided ideal generated by the element
\[
a=2+2p+r+2x+3r^2+rx.
\]
Then, $\{2,r,1-p\}\subset (a)$.
\end{lemma}
\begin{proof}
First, one can verify that $2+r=ba$, where
\[
b=1-x-r^2-2rx,
\]
so that $2+r\in(a)$. Next, one has $r=ca+(2+r)x$, where
\[
c=-2+p+3r+4x-rx,
\]
so that $r\in (a)$, and thus $2\in (a)$. Finally, remarking that
\[
1-p=2x+r^2x,
\]
we conclude that $1-p\in (a)$.
\end{proof}
\begin{lemma}\label{lem:20}
Let $p,q,x$ be three elements of a ring $R$ satisfying the three identities
\begin{eqnarray}
\label{eq:400}p&=&x-x^2,\\
\label{eq:410}q&=&px-3p-x+1,\\
\label{eq:420}pq&=&q.
\end{eqnarray}
Then, the element $p$ is invertible if and only if
\begin{equation}\label{eq:430}
2q+p^2+4p-1=0.
\end{equation}
\end{lemma}
\begin{proof}
First, multiplying by $q$ identity~\eqref{eq:410} and simplifying the right hand side by the use of equation~\eqref{eq:420}, we see that
\begin{equation}\label{eq:435}
q^2=-2q.
\end{equation}
Next,
excluding $x$ from identities~\eqref{eq:400} and ~\eqref{eq:410}, we obtain the identity
\[
q^2+(5p-1)q+p(p^2+4p-1)=0
\]
which, due to equations~\eqref{eq:420} and \eqref{eq:435}, simplifies to
\begin{equation}\label{eq:440}
2q+p(p^2+4p-1)=0.
\end{equation}
Now, if $p$ is invertible, then equations~$\eqref{eq:440}$, $\eqref{eq:420}$  imply equation~$\eqref{eq:430}$. Conversely, if $\eqref{eq:430}$ is true, then combining it with $\eqref{eq:440}$, we obtain the polynomial identity
\[
(p-1)(p^2+4p-1)=0
\]
which implies invertibility of $p$ with the inverse
\(
p^{-1}=5-3p-p^2.
\)
\end{proof}
\begin{proposition}
(i) There exists a unique ring homomorphism $\phi\colon \hat R_{G,H}\to S$ such that
\[
s_{a_1}\mapsto x,\quad s_{a_2}\mapsto x+r,\quad v_{a_1}\mapsto 0,\quad v_{a_2}\mapsto 1.
\]

(ii) For any weakly special representation $\rho$, considered up to equivalence, the ring homomorphism $f_\rho$ factorizes through $\phi$, i.e. there exists a unique ring homomorphism $h_\rho\colon S\to R$ such that
$f_\rho=h_\rho\circ\phi$.

(iii) The set $W(G,H)$ consists of two elements.
\end{proposition}
\begin{proof}
(i) This is a straightforward verification.

(ii)
Let $R$ be a non-trivial ring and
\[
\rho\colon G\ni g\mapsto (\alpha(g),\beta(g))\in R^*\ltimes R,
\] a weakly special representation of $G$. Then, the element $\xi=\beta(a_2)$ is invertible, since  otherwise $\beta^{-1}(0)=G$. By replacing $\rho$ with an equivalent representation, we can assume that $\xi=1$.

Denote
\[
x_i=\alpha(a_i),\quad w_x=\alpha(w_a),\quad \bar w_x=\alpha(\bar w_a).
\]
Note that
\[
0=\beta(a_2a_2^{-1})=\beta(a_2^{-1})+x_2^{-1}\Leftrightarrow \beta(a_2^{-1})=-x_2^{-1}.
\]
From the definitions of $w_a$ and $\bar w_a$ we have
\begin{equation}\label{eq:100}
\beta(w_a)=\beta(a_2^{-1}a_1a_2a_1^{-1})
=\beta(a_1a_2a_1^{-1})+\beta(a_2^{-1})x_1x_2x_1^{-1}
=x_1^{-1}-w_x,
\end{equation}
\begin{equation}\label{eq:110}
\beta(\bar w_a)=\beta(a_1^{-1}a_2a_1a_2^{-1})=-x_2^{-1}+x_1x_2^{-1}=-(1-x_1)x_2^{-1}.
\end{equation}
From the relation $a_1w_a=w_aa_2$ we obtain an identity
\[
\beta(w_a)=\beta(a_1w_a)=\beta(w_aa_2)=\beta(w_a)x_2+1
\]
which implies that $\beta(w_a)\ne 0$, and thus $\beta(w_a)$ is invertible with
\[
\beta(w_a)^{-1}=1-x_2.
\] Invertibility
of $\beta(\bar w_a)$ follows from the equation
\[
0=\beta(w_a\bar w_a)=\beta(w_a)\bar w_x+\beta(\bar w_a).
\]
Compatibility of the latter equations with the formulae~\eqref{eq:100}, \eqref{eq:110} implies the following anticommutation relations
\begin{equation}\label{eq:120}
x_i^{-1}x_j+x_jx_i^{-1}=x_i^{-1}+x_j-1,\quad \{ i,j\}=\{1,2\}.
\end{equation}
Indeed,
\begin{multline*}
x_1x_2^{-1}=(1-x_2)\beta(w_a)x_1x_2^{-1}=(1-x_2)(x_1^{-1}-w_x)x_1x_2^{-1}\\
=(1-x_2)(x_1^{-1}-x_2^{-1}x_1x_2x_1^{-1})x_1x_2^{-1}
=(1-x_2)(x_2^{-1}-x_2^{-1}x_1)\\
=(x_2^{-1}-1)(1-x_1)=x_1+x_2^{-1}-1-x_2^{-1}x_1,
\end{multline*}
and
\begin{multline*}
x_1^{-1}x_2=x_1^{-1}x_2x_1x_2^{-1}x_2x_1^{-1}=\bar w_xx_2x_1^{-1}
=-\beta(w_a)^{-1}\beta(\bar w_a)x_2x_1^{-1}\\
=(1-x_2)(1-x_1)x_2^{-1}x_2x_1^{-1}=(1-x_2)(x_1^{-1}-1)=x_2+x_1^{-1}-1-x_2x_1^{-1}.
\end{multline*}
By using relations~\eqref{eq:120}, one obtains the following formula:
\[
w_x^{-1}=(x_2^{-1}-1)(1-x_1-x_2)x_1^{-1}.
\]
Indeed,
\begin{multline*}
w_x^{-1}=x_1x_2^{-1}x_1^{-1}x_2=(x_1+x_2^{-1}-1-x_2^{-1}x_1)x_1^{-1}x_2
=(x_2^{-1}-1)(1-x_1)x_1^{-1}x_2\\
=(x_2^{-1}-1)(x_1^{-1}x_2-x_2)=(x_2^{-1}-1)(x_1^{-1}-1-x_2x_1^{-1})=(x_2^{-1}-1)(1-x_1-x_2)x_1^{-1}.
\end{multline*}
This formula implies the following equivalences:
\[
x_2w_x^{-1}=w_x^{-1}x_1\ \Leftrightarrow\ \\ x_2(1-x_1-x_2)=(1-x_1-x_2)x_1\ \Leftrightarrow\ x_2(1-x_2)=(1-x_1)x_1.
\]
The latter identity can be equivalently rewritten in the form
\begin{equation}\label{eq:130}
x_1x_{21}+x_{21}x_1=x_{21}-x_{21}^2,
\end{equation}
where
\[
x_{21}=x_2-x_1,
\]
and the same identity implies that there exists an invertible element $z$ such that
\[
x_i(1-x_i)=z,\quad i\in\{1,2\}.
\]
Evidently, $z$ commutes with both $x_1$ and $x_2$. We have the following formulae for the inverses of $x_i$:
\[
x_i^{-1}=(1-x_i)z^{-1},
\]
substitute them into equations~\eqref{eq:120}, and rewrite the result as the following two equations
\begin{equation}\label{eq:140}
x_1x_{21}+x_{21}x_1=x_{21}+3z+x_1-zx_1-1,
\end{equation}
and
\begin{equation}\label{eq:150}
(z-1)x_{21}=0.
\end{equation}
Compatibility of equations~\eqref{eq:130} and \eqref{eq:140} gives one more identity
\begin{equation}\label{eq:160}
zx_1+1=3z+x_1+x_{21}^2.
\end{equation}
Finally, applying lemma~\ref{lem:20} to elements $p,x_{21}^2,x_1$, we obtain one more identity
\[
2x_{21}^2+p^2+4p-1=0.
\]
Now, we can easily see that the mapping
\[
h_\rho(p)=z,\quad h_\rho(r)=x_{21},\quad h_\rho(x)=x_1,
\]
in a unique way extends to a ring homomorphism $h_\rho\colon S\to R$, and one has the factorization formula $f_\rho=h_\rho\circ\phi$.

(iii)
We remark that our pair $(G,H)$ and any weakly special representation $\rho$ verify the conditions of proposition~\ref{pro:100}. For example, let us check that there exists an element $t_0\in \alpha^{-1}(1)$ such that $\beta(t_0)$ is invertible. In the case if $x_{21}=0$, we can choose $t_0=a_1^{-1}a_2$ for which $\beta(t_0)=1$, while for $x_{21}\ne 0$ we can choose $t_0=w_a\bar w_a^{-1}$ for which
\[
\beta(t_0)=2+2z+x_{21}+2x_1+3x_{21}^2+x_{21}x_1.
\]
Due to lemma~\ref{lem:15}, the latter element is non-zero and therefore invertible.
Thus, taking into account the parts~(i) and (ii), to identify the elements of the set $W(G,H)$, it is enough to find all minimal quotients of the ring $S$ which can be embedded into rings the way that all non-zero elements become invertible. In particular, the quotient rings must be integral domains.

We have two relations in $S$ which indicate existence of zero-divisors:
\[
r(1-p)=0,\quad r(2+r^2)=0.
\]
We have the following mutually excluding possibilities for minimal quotients which remove these relations:
\begin{enumerate}
\item $r=0$;
\item $p=1$, $r^2=-2\ne0$.
\end{enumerate}

 The case (1) gives the quotient ring
 \[
 S/(r)\simeq \mathbb{Z}[t]/(\Delta_{4_1}(t)), \quad \Delta_{4_1}(t)=t^2-3t+1,\quad x\mapsto t,\  p\mapsto 1-2t.
 \] This is a commutative integral domain, and its localization at the set of non-zero elements coincides with the field of fractions $\mathbb{Q}[t]/(\Delta_{4_1}(t))$.

 The case (2) gives the quotient ring
 \[
 S/(1-p, 2+r^2)\simeq \mathbb{Z}\langle r,x\rangle/(1-x+x^2,rx+xr-2-r),
 \]
 which is isomorphic to the ring of Hurwitz integral quaternions
 \[
 H=\left\{a+bi+cj+dk\vert\ (a,b,c,d)\in\mathbb{Z}^4\sqcup(2^{-1}+\mathbb{Z})^4\right\},
 \]
 where
 \[
 1+i^2=1+j^2=ij+ji=0,\quad k=ij,
 \]
the isomorphism being given by the map
\[
r\mapsto i+j,\quad x\mapsto 2^{-1}(1-i-j+k),
\]
and the inverse map
\[
i\mapsto rx-1,\quad j\mapsto xr-1.
\]
 Thus,  $S/(1-p, 2+r^2)$ admits an embedding into a (non-commutative) division ring of rational quaternions.
 \end{proof}
Let $\rho_i$, $i\in\{1,2\}$ represent the two elements of $W(G,H)$.
Corresponding to $\rho_1$ the group pair $(G_1,H_1)$ admits a presentation
 \[
 G_1=\langle s,t_0,t_1\vert\  t_0t_1=t_1t_0,\ st_1=t_0s,\ t_1st_0=t_0^3s\rangle\simeq \mathbb{Z}\ltimes\mathbb{Z}^2,\quad H_1=\langle s\rangle\simeq\mathbb{Z}
 \]
 with the projection homomorphism
\[
\pi_1\colon (G,H)\to (G_1,H_1),\quad  a_1\mapsto s,\ a_2\mapsto st_0,
\]
whose kernel
\(
\mathrm{ker}(\pi_1)
\)
can be shown to be generated by the longitude $l=w_a\bar w_a$. In other words, we have a group isomorphism
\[
G_1\simeq \langle a_1,a_2\vert\ a_1w_a=w_aa_2,\ w_a\bar w_a=1\rangle,\quad w_a=a_2^{-1}a_1a_2a_1^{-1},\ \bar w_a=a_1^{-1}a_2a_1a_2^{-1}.
\]

To describe the group pair $(G_2,H_2)$, corresponding to $\rho_2$, consider the following representation of $G$ in $SL(4,\mathbb{Z})$:
 \[
 a_1\mapsto \begin{pmatrix}
 0 &1&0 &0\\
 -1&1&0 &0\\
 0 &0&0 &1\\
 0 &0&-1&1
 \end{pmatrix},\quad
 a_2\mapsto \begin{pmatrix}
 0 &0&1 & 0\\
 0 &1&1 &-1\\
 -1&0&1 & 0\\
 -1&1&1 & 0
 \end{pmatrix},
 \]
whose kernel is generated, for example, by the element $a_1^{-1}a_2a_1^2a _2$, and
the corresponding right action of $G$ on $\mathbb{Z}^4$ (given by the multiplication of integer row vectors by above matrices). Then, the group $G_2$ is the quotient group of the semidirect product
$G\ltimes\mathbb{Z}^4$ by the relation $a_1 (1,0,0,0)=a_2a_1^2a_2$, where we identify $G$ and $\mathbb{Z}^4$ as subgroups of $G\ltimes\mathbb{Z}^4$, while the subgroup $H_2=\langle a_1\rangle\simeq\mathbb{Z}_6$.

\end{example}
\section{Presentations of $\Delta$-groupoids}\label{sec:pres}
\subsection{The tetrahedral category}

For any non-negative integer $n\ge0$ we identify the symmetric group
$\mathbb{S}_n$ as the sub-group of all permutations (bijections) of the set of non-negative integers $\mathbb{Z}_{\ge0}$  acting identically on the subset $\mathbb{Z}_{\ge n}\subset \mathbb{Z}_{\ge0}$. This interpretation fixes  a canonical inclusion $\mathbb{S}_m\subset\mathbb{S}_{n}$ for any pair $m\le n$. The standard generating set $\{s_i=(i-1,i)\vert\ 1\le i< n\}$ of $\mathbb{S}_n$ is given by elementary transpositions of two consecutive integers. Later on, it will be convenient to use the inductive limit
\[
\mathbb{S}_\infty=\lim_{\longrightarrow}\mathbb{S}_n=\cup_{n\ge0}\mathbb{S}_n.
\]
We also denote by $\myGset{n}$ the category of $\mathbb{S}_n$-sets, i.e. sets with a left $\mathbb{S}_n$-action and $\mathbb{S}_n$-equivariant maps as morphisms.

We remark that in any
$\Delta$-groupoid $G$, its distinguished generating set $H$ is an $\mathbb{S}_3$-set given by the identifications $s_1=j$ and $s_2=i$, while the set $V\subset H^{2}$ of $H$-composable pairs is an $\mathbb{S}_4$-set given by the rules
\begin{equation}\label{eq:50}
s_1(x,y)=(j(x),j(k(x)j(y))),\quad s_2(x,y)=(i(x),xy),\quad
s_3(x,y)=(xy,i(y)),
\end{equation}
and the projection map to the first component
\[
\mathrm{pr}_1\colon V\ni (x,y)\mapsto x\in H
\]
being $\mathbb{S}_3$-equivariant.

Let $R_{43}\colon\myGset{4}\to \myGset{3}$ be the restriction (to the subgroup) functor. Consider the comma category\footnote{see \cite{ML} for a general definition of a comma category.} $(R_{43}\downarrow\myGset{3})$ of $\mathbb{S}_3$-equivariant maps $a\colon R_{43}(V_a)\to I_a$, for some $\mathbb{S}_4$-set  $V_a$, and $\mathbb{S}_3$-set $I_a$. Call it the \emph{tetrahedral} category. An object of this category will be called \emph{tetrahedral object}. A morphism between two  tetrahedral objects $f\colon a\to b$, called  \emph{tetrahedral morphism}, is a pair $(p_f,q_f)$ where $p_f\in \myGset{4}(V_a,V_b)$ and  $q_f\in \myGset{3}(I_a,I_b)$ are such that $bR_{43}(p_f)=q_fa$. Taking into account the remarks above, to each $\Delta$-groupoid $G$ we can associate a tetrahedral object $C(G)=\mathrm{pr}_1\colon R_{43}(V)\to H$.
If $f\colon G \to G'$ is a morphism of $\Delta$-groupoids, then the pair
\[
C(f)=(f\times f\vert_{V},f\vert_{H})
\]
 is a tetrahedral morphism such that $C(fg)=C(f)C(g)$. Thus, we obtain a functor
$C\colon \myK\to (R_{43}\downarrow\myGset{3})$.

 A $\Delta$-groupoid $G$ is called \emph{finite} if its distinguished set $H$ is finite (note that $G$ itself can be an infinite groupoid). Let $\myK_{\mathrm{fin}}$ be the full subcategory of finite $\Delta$-groupoids, and
$(R_{43}\downarrow\myGset{3})_{\mathrm{fin}}$ the full subcategory of \emph{finite} tetrahedral objects (i.e. $a$'s with finite $V_a$ and $I_a$). Then, the functor $C$ restricts to a functor
\[
C_{\mathrm{fin}}\colon \myK_{\mathrm{fin}}\to (R_{43}\downarrow\myGset{3})_{\mathrm{fin}}.
\]
\begin{proposition}
The functor $C_{\mathrm{fin}}$ admits a left adjoint
\[
C_{\mathrm{fin}}'\colon (R_{43}\downarrow\myGset{3})_{\mathrm{fin}}\to\myK_{\mathrm{fin}}
\]
which verifies the identities
 \[
 C_{\mathrm{fin}}' C_{\mathrm{fin}}C_{\mathrm{fin}}'=C_{\mathrm{fin}}',\quad
  C_{\mathrm{fin}} C_{\mathrm{fin}}'C_{\mathrm{fin}}=C_{\mathrm{fin}}.
 \]
\end{proposition}
\begin{proof}
 Let  $a$ be an arbitrary finite tetrahedral object.
Consider a map
\[
\tau_a\colon V_a\to I_a^{2},\quad
\tau_a(v)=
(a(v),a((321)(v))),\quad \forall v\in V_a
\]
Let $\mathcal{R}$ be the minimal $\mathbb{S}_4$-equivariant equivalence relation on $V_a$ generated by the set
\[
\bigcup_{x\in I_a^{2}}\tau_a^{-1}(x)^{2}
\] and $\mathcal{R}^a$, the $a$-image of $\mathcal{R}$ which is necessarily
an $\mathbb{S}_3$-equivariant equivalence relation on $I_a$. Let  $p_a\colon V_a\to V_a/\mathcal{R}$ and $q_a\colon I_a\to I_a/\mathcal{R}^a$ be the canonical equivariant projections on the quotient sets. Then, there exists a unique tetrahedral object $a_1$ such that the pair $(p_a,q_a)$ is a tetrahedral morphism from $a$ to $a_1$. Iterating this procedure, we obtain a sequence of tetrahedral morphisms
\[
a\to a_1\to a_2\to\cdots
\]
which, due to finiteness of $a$, stabilizes  in a finite number of steps to a tetrahedral object $\tilde{a}$.
It is characterized  by the property that the map $\tau_{\tilde{a}}$ is an injection, so that we can identify
the set $V_{\tilde{a}}$ with its $\tau_{\tilde{a}}$-image in $I_{\tilde{a}}^{2}$.

For any $(x,y)\in V_{\tilde{a}}$ denote $\tilde{a}(s_3(x,y))=xy$ and call it \emph{product}. Then, the action of the group $\mathbb{S}_4$ on $V_{\tilde{a}}$ is given by the formulae ~\eqref{eq:50}, where $i(x)=s_2(x)$, $j(x)=s_1(x)$, $k(x)=(02)(x)$, and the consistency conditions imply the following properties of the product:
\begin{equation}\label{eq:40}
i(xy)=i(y)i(x),\quad k(xy)=k(k(x)j(y))k(y),\quad i(x)(xy)=(yx)i(x)=y.
\end{equation}
However, this product is not necessarily associative, and to repair that, we consider
the minimal $\mathbb{S}_3$-equivariant equivalence relation $\mathcal{S}$
on $I_{\tilde{a}}$ generated by the relations $x(yz)\sim (xy)z$, where all products are supposed to make sense. Let $\pi_{\tilde{a}}\colon I_{\tilde{a}}\to I_{\tilde{a}}/\mathcal{S}$ be  the canonical $\mathbb{S}_3$-equivariant projection on the quotient set. Then, clearly, $\tilde{a}'=\pi_{\tilde{a}}\circ\tilde{a}$ is another (finite) tetrahedral object
with $V_{\tilde{a}'}=V_{\tilde{a}}$ and $ I_{\tilde{a}'}=I_{\tilde{a}}/\mathcal{S}$, and with canonically associated tetrahedral morphism $(\mathrm{id}_{\tilde{a}},\pi_{\tilde{a}})\colon \tilde{a}\to \tilde{a}'$.
Applying the "tilde" operation to $\tilde{a}'$, we obtain a composed
morphism
 \[
 (p_{\tilde{a}'},q_{\tilde{a}'})\circ(\mathrm{id}_{\tilde{a}},\pi_{\tilde{a}})\circ(p_a,q_a)\colon a\to \hat{a}=\widetilde{{\tilde{a}}'}
 \]
  Again, the iterated sequence of such morphisms
\[
a\to\hat{a}\to \hat{\hat{a}}\to\cdots
\]
stabilizes in finite number of steps to an object $\dot{a}$ with $V_{\dot{a}}\subset I_{\dot{a}}^{2}$ and the associative product $xy=\dot{a}(s_3(x,y))$ satisfying the relations~\eqref{eq:40}.

Now,  let $\mathcal{T}$ be the minimal equivalence relation on $I_{\dot{a}}$ generated by the set $(i\times \mathrm{id}_{I_{\dot{a}}})(V_{\dot{a}})$. Denote by
\[
\mathrm{dom}\colon I_{\dot{a}}\to N_{\dot{a}}=I_{\dot{a}}/\mathcal{T}
 \]
 the canonical projection to the quotient set. Define another map
\[
\mathrm{cod}=\mathrm{dom}\circ i\colon I_{\dot{a}}\to N_{\dot{a}}.
\]
In this way, we obtain a graph (quiver) $\Gamma_a$ with the set of arrows $I_{\dot{a}}$, the set of nodes
$N_{\dot{a}}$, and the domain (source)  and the codomain (target) functions $\mathrm{dom}(x)$, $\mathrm{cod}(x)$. Thus, we obtain a finite  $\Delta$-groupoid
with the presentation
\[
C_{\mathrm{fin}}'(a)=\langle \Gamma_a \vert\  x\circ y=xy\ \mathrm{if}\ (x,y)\in  V_{\dot{a}}\rangle
\]
whose distinguished generating set is given by
\(
I_{\dot{a}}.
\)
\end{proof}
\subsection{$\Delta$-complexes}
Let
\[
\Delta^n=\{(t_0,t_1,\ldots,t_n)\in [0,1]^{n+1}\vert\ t_0+t_1+\cdots+t_n=1\},\quad n\ge 0
\]
be the standard $n$-simplex with face inclusion maps
\[
\delta_m\colon \Delta^{n}\to\Delta^{n+1},\quad 0\le m\le n
\] defined by
\[
\delta_m(t_0,\ldots,t_{n})=(t_0,\ldots, t_{m-1},0,t_{m},\ldots,t_n)
\]
A (simplicial) cell in a topological space $X$ is a continuous map
\[
f\colon \Delta^n\to X
\]
 such that the restriction of $f$ to the interior of $\Delta^n$ is an embedding. On the set of cells $\Sigma(X)$ we have the dimension function
 \[
 d\colon\Sigma(X)\to\mathbb{Z},\quad (f\colon \Delta^n\to X) \mapsto n.
 \]
A.~Hatcher in \cite{Hatcher} introduces $\Delta$-complexes as a generalization of simplicial complexes. A $\Delta$-complex structure on a topological space $X$ can be defined as a pair $(\Delta(X), \partial)$, where $\Delta(X)\subset \Sigma(X)$ and $\partial$ is a set of maps
\[
\partial=\left\{\left.\partial_n\colon d\vert_{\Delta(X)}^{-1}(\mathbb{Z}_{\ge \max(1,n)})\to d\vert_{\Delta(X)}^{-1}(\mathbb{Z}_{\ge {n-1}})\ \right\vert\ n\ge0 \right\}
\]
such that:
\begin{itemize}
\item[(i)] each point of $X$ is in the image of exactly one restriction, $\alpha\vert_{{\stackrel{\circ}{\Delta}}^{n}}$ for $\alpha\in \Delta(X)$;

\item[(ii)] $\alpha\circ\delta_m=\partial_m\alpha$;

\item[(iii)]a set $A\subset X$ is open iff $\alpha^{-1}(A)$ is open for each $\alpha\in \Delta(X)$.
\end{itemize}

Clearly, any $\Delta$-complex is a CW-complex.

\subsection{Tetrahedral objects from $\Delta$-complexes}
We  associate a tetrahedral object $a_X$ to a $\Delta$-complex $X$ as follows:
\[
V_{a_X}=\mathbb{S}_{4}\times\Delta^3(X) ,\quad I_{a_X}= \mathbb{S}_{3}\times\Delta^2(X),
\]
which  are  $\mathbb{S}_{n}$-sets with the groups acting by left multiplications on the first components, and
\[
a_X((i3),x)=(g_i,\partial_{i}x),  \  i\in\{0,1,2,3\}, \quad g_0=(012),\  g_1=(12),\ g_2=g_3=1,
\]
where $(33)=1$, and the value $a_X(g,x)$ for $g\ne (i3)$ is uniquely deduced from its $\mathbb{S}_{3}$-equivariance property.

\subsection{Ideal triangulations of knot complements}
A particular class of three dimensional $\Delta$-complexes arises as ideal triangulations of knot complements. A simple caculation shows that in any ideal triangulation there are equal number of edges and tetrahedra and twice as many faces. For example, the two simplest non-trivial knots (the trefoil and the figure-eight) admit ideal triangulations with only two tetrahedra $\{u,v\}$, four faces $\{a,b,c,d\}$ and two edges $\{p,q\}$, the difference being in the gluing rules. Using the notation $(x|\partial_0x,\partial_1x,\ldots,\partial_nx)$ these examples read as follows.
\begin{example}[The trefoil knot] The gluing rules are given by the list
\[
(u|a,b,c,d),\ (v|d,c,b,a),\ (a|p,p,p),\ (b|p,q,p),\ (c|p,q,p),\ (d|p,p,p).
\]
The associated $\Delta$-groupoid $G$ is freely generated by the quiver (oriented graph) consisting of two vertices $A$ and $B$ and two arrows
$x$ and $y$ with
\[
\mathrm{dom}(x)=\mathrm{cod}(x)=\mathrm{dom}(y)=A,\quad \mathrm{cod}(y)=B,
\]
with the distinguished subset
\[
H=\{x^{\pm1},x^{\pm2},y^{\pm1},(xy)^{\pm1}\},\quad j\colon x\mapsto x^{-1},\ x^2\mapsto y,\ x^{-2}\mapsto xy.
\]
One can show that
\[
A'G\simeq B'G\simeq \mathbb{Z}[t,3^{-1}]/(\Delta_{3_1}(t)),\quad \Delta_{3_1}(t)=t^2-t+1.
\]
\end{example}
\begin{example}[The figure-eight  knot] The gluing rules are given by the list
\[
(u|a,b,c,d),\ (v|c,d,a,b),\ (a|p,q,q),\ (b|p,p,q),\ (c|q,p,p),\ (d|q,q,p).
\]
In this case, there are no non-trivial identifications in the corresponding $\Delta$-groupoid $G$, and the ring $A'G$ is isomorphic to certain localization of the ring $S$ from Example~\ref{ex:4_1}, while
\[
B'G\simeq\mathbb{Z}\langle u^{\pm1},v^{\pm1},w^{\pm1}\vert\ u(u+1)=w,\ v(v+1)=w^{-1}\rangle.
\]
Notice that in the Abelian case the defining relations of the ring $B'G$ can be identified with the hyperbolicity equations for the figure-eight knot.
\end{example}
\section{Homology of $\Delta$-groupoids}\label{sec:hom}
Given a $\Delta$-groupoid $G$ with the distinguished generating subset $H$.  We define recursively the following sequence of sets:
\(
V_{-1}=\{*\}
\)
(a singleton or one element set),
\(
V_0=\pi_0(G)
\) (the set of connected components of $G$),
\( V_1=\mathrm{Ob}(G)
\)
(the set of objects or identities of $G$),
\(
V_2=H,
\)
$V_3$ is the set of all $H$-composable pairs, while  $V_{n+1}$ for $n>2$ is the collection of all $n$-tuples $(x_1,x_2,\ldots,x_n)\in H^n$,
satisfying the following conditions:
\[
(x_i,x_{i+1})
\in V_3, \quad 1\le i\le n-1,
\]
 and
 \[
 \partial_i(x_1,\ldots,x_n)\in V_{n},\quad 1\le i\le n+1,
 \] where
\begin{equation}\label{eq:900}
\partial_i(x_1,\ldots,x_n)=\left\{\begin{array}{ll}
(x_2,\ldots,x_{n}),& i=1;\\
 (x_1,\ldots,x_{i-2},x_{i-1}x_i,x_{i+2},\ldots,x_n),& 2\le i\le n;\\
 (x_1,\ldots,x_{n-1}),& i=n+1.
 \end{array}
 \right.
\end{equation}
\begin{remark}
If $N(G)$ is the nerve of $G$, then the system $\{V_n\}_{n\ge 1}$ can be defined as the maximal system of subsets $V_{n+1}\subset N(G)_{n}\cap H^n$, $n=0,1,\ldots$, (with the identification $H^0=\mathrm{Ob}(H)$) closed under the face maps of the simplicial set $N(G)$.
\end{remark}
 For any $H$-composable pair $(x,y)$ we introduce a binary operation
 \begin{equation}\label{eq:*oper}
 x*y=j(k(x)j(y))=jk(x)j(xy)=ij(x)j(xy).
 \end{equation}
 \begin{proposition}
 For any integer $n\ge 2$, if $(x_1,\ldots, x_n)\in V_{n+1}$, then the  $(n-1)$-tuple
 \begin{equation}\label{eq:800}
\partial_0(x_1,\ldots,x_{n})=(y_1,\ldots,y_{n-1}),\quad y_i=z_i*x_{i+1},\quad z_i=x_1x_2\cdots x_i,
\end{equation}
 is an element of $V_{n}$.
 \end{proposition}
 \begin{proof}
 We proceed by induction on $n$. For $n=2$ the statement is evidently true. Choose an integer $k\ge3$ and assume the statement is true for $n=k-1$. Let us prove that it is also true for $n=k$. Taking into account the formula
 \[
 y_i=ij(z_i)j(z_{i+1}),\quad 1\le i\le k-1,
 \]
 we see that for any $1\le i\le k-2$, the pair $(y_i,y_{i+1})$ is $H$-composable with the product
 \[
 y_iy_{i+1}=ij(z_i)j(z_{i+2})=ij(z_i)j(z_ix_{i+1}x_{i+2})=z_i*(x_{i+1}x_{i+2}).
 \]
 Now, the $(k-2)$-tuples
 \begin{multline*}
 (y_1,\ldots,y_{i-1}, y_iy_{i+1},y_{i+2},\ldots,y_{k-1})\\
 =\partial_{0}(x_1,\ldots,x_i,x_{i+1}x_{i+2},x_{i+3},\ldots,x_{k}),\quad
 1\le i\le k-2,
 \end{multline*}
 as well as
 \[
 (y_1,\ldots,y_{k-2})=\partial_{0}(x_1,\ldots,x_{k-1}),\quad  (y_2,\ldots,y_{k-1})=\partial_{0}(x_1x_2,\ldots,x_{k-1}),
 \]
 are all in $V_{k-1}$ by the induction hypothesis, and thus $(y_1,\ldots,y_{k-1})\in V_{k}$.
 \end{proof}
Definitions~\eqref{eq:900},\eqref{eq:800} also make sense for $n=2$, and, additionally, we extend them for three more values $n=-1,0,1$ as follows:
\begin{itemize}
\item[$(n=-1)$] as $V_{-1}$ is a terminal object in the category of sets, there are no other choices but one for the map $\partial_0\vert_{V_0}$;
    \item[$(n=0)$] if for $A\in V_1$ we denote $[A]$  the connected component of $G$ defined by $A$, then
 \begin{equation}\label{eq:dv1}
 \partial_0A=[A^*],\ \partial_1A=[A];
 \end{equation}
 \item[$(n=1)$] using the  domain (source) and the codomain (target) maps of the groupoid $G$ (viewed as a category), we define
\begin{equation}\label{eq:dv2}
\partial_0x=\mathrm{cod}(j(x)),\quad \partial_1x=\mathrm{cod}(x),\quad \partial_2x=\mathrm{dom}(x),\quad x\in V_2.
\end{equation}
\end{itemize}

 For any $n\ge-1$, let $B_n=\mathbb{Z}V_n$ be the abelian group freely generated by the elements of $V_n$. Define also $B_n=0$, if $n<-1$.
 We extend linearly the maps $\partial_i$, $i\in\mathbb{Z}_{\ge0}$ to a family of group endomorphisms
  \[
    \partial_i\colon B\to B,\quad B=\oplus_{n\in\mathbb{Z}}B_n
  \]
so that $\partial_i\vert_{B_n}=0$ if $i>n$. Then, the formal linear combination
 \[
 \partial=\sum_{i\ge0}(-1)^i\partial_i
\]
is a well defined endomorphism of the group $B$ such that  $\partial B_n\subset B_{n-1}$.
\begin{proposition}
The group $B$ is a chain complex with differential $\partial$ and grading operator
\(
p\colon B\to B,\quad p\vert_{B_n}=n.
\)
\end{proposition}
\begin{proof}
It is immediate to see that the group homomorphism $\partial'=\partial_0-\partial$ is a restriction of the differential of the standard integral (augmented) chain complex of the nerve $N(G)$ so that $\partial'^2=0$.
Now, due to the latter equation, the equation $\partial^2=0$ is equivalent to the equation
\[
\partial_0^2=\partial_0\partial'+\partial'\partial_0
\]
which can straightforwardly be checked on basis elements.
\end{proof}
\begin{proposition}
There exists a unique sequence of set-theoretical maps
\[
\delta_i\colon \mathbb{S}_\infty\to\mathbb{S}_\infty,\quad i\in\mathbb{Z}_{\ge0},
\]
such that for any $j\in\mathbb{Z}_{>0}$,
\begin{equation}\label{eq:cocycle}
\delta_i(s_j)=\left\{\begin{array}{cl}
s_{j-1}&\mathrm{if}\ i<j-1;\\
1&\mathrm{if}\ i\in\{j-1,j\};\\
s_j&\mathrm{if}\ i>j,
\end{array}
\right.
\end{equation}
and
\begin{equation}\label{eq:prod-prop}
\delta_i(gh)=\delta_i(g)\delta_{g^{-1}(i)}(h),\quad \forall g,h\in\mathbb{S}_\infty.
\end{equation}
\end{proposition}
\begin{proof}
From the identity
\[
\delta_i(g)=\delta_i(1g)=\delta_i(1)\delta_i(g)
\]
it follows immediately that $\delta_i(1)=1$ for any $i\in\mathbb{Z}_{\ge0}$.
To prove the statement it is enough to check consistency of the defining relations of the group $\mathbb{S}_\infty$ with formulae~\eqref{eq:cocycle}, \eqref{eq:prod-prop}, i.e. the equations
\begin{multline*}
\delta_i(s_j)\delta_{s_j(i)}(s_k)=\delta_i(s_k)\delta_{s_k(i)}(s_j),\quad |j-k|>1,\\
 \delta_i(s_j)\delta_{s_j(i)}(s_{j+1})\delta_{s_{j+1}s_j(i)}(s_j)
=\delta_i(s_{j+1})\delta_{s_{j+1}(i)}(s_{j})\delta_{s_js_{j+1}(i)}(s_{j+1}),\\
\delta_i(s_j)\delta_{s_j(i)}(s_j)=1,\quad
\forall i\in\mathbb{Z}_{\ge0},\ \forall j\in\mathbb{Z}_{>0}.
\end{multline*}
This is a straightforward verification.
\end{proof}
\begin{remark}
For any $n\ge 1$, we have $\delta_i(\mathbb{S}_{n})\subset \mathbb{S}_{n-1}$, $0\le i\le n$.
\end{remark}
\begin{proposition}
For any $n\ge-1$ the set $V_n$ has a unique canonical (i.e. independent of particularities of a $\Delta$-groupoid) structure of an  $\mathbb{S}_{n+1}$-set such that
\begin{equation}\label{eq:intertwining}
\partial_i\circ g=\delta_i(g)\circ\partial_{g^{-1}(i)},\quad 0\le i\le n,\ g\in\mathbb{S}_{n+1}.
\end{equation}
\end{proposition}
\begin{proof}
For $n\le0$ the statement is trivial. For $n>0$ and $g\in\{s_1,\ldots,s_n\}$, equations~\eqref{eq:intertwining} take the form
\begin{equation}\label{eq:int-gen}
\partial_i\circ s_j=\left\{
\begin{array}{cl}
    s_{j-1}\circ\partial_i&\mathrm{if}\ i<j-1;\\
    \partial_j&\mathrm{if}\ i=j-1;\\
    \partial_{j-1}&\mathrm{if}\ i=j;\\
    s_j\circ\partial_i&\mathrm{if}\ i>j.
\end{array}
\right.
\end{equation}

In the case $n=1$, together with equations~\eqref{eq:dv1}, these imply that for any $A\in V_1$
\[
\partial_i(s_1(A))=\partial_i(A^*),\quad i\in\{0,1\}.
\]
The only canonical solution to these equations is of the form
\begin{equation}\label{eq:perm-n=1}
s_1(A)=A^*,\quad \forall A\in V_1.
\end{equation}

In the case  $n=2$, equations~\eqref{eq:int-gen}, \eqref{eq:dv2}, and \eqref{eq:perm-n=1} imply that
\[
\partial_m\circ s_1=\partial_m\circ j,\quad \partial_m\circ s_2=\partial_m\circ i,\quad 0\le m\le 2,
\]
with only one canonical solution of the form
\[
s_1(x)=j(x),\quad s_2(x)=x^{-1},\quad \forall x\in V_2.
\]

In the case $n>2$, one can show that the following formula constitutes a solution to system~\eqref{eq:int-gen}:
\begin{multline}
s_i(x_1,\ldots,x_{n-1})\\
=\left\{
\begin{array}{ll}
(j(x_1),x_1*x_2,(x_1x_2)*x_3,\ldots,(x_1\cdots x_{n-2})*x_{n-1})&\mathrm{if}\ i=1;\\
(x_1^{-1},x_1x_2,x_3,\ldots,x_{n-1})&\mathrm{if}\ i=2;\\
(x_1,\ldots,x_{i-3},x_{i-2}x_{i-1},x_{i-1}^{-1},x_{i-1}x_i,x_{i+1},\ldots,x_{n-1})&\mathrm{if}\ 2<i<n;\\
(x_1,\ldots,x_{n-3},x_{n-2}x_{n-1},x_{n-1}^{-1})&\mathrm{if}\ i=n,
\end{array}
\right.
\end{multline}
where we use the binary operation~\eqref{eq:*oper}. Let us show that there are no other (canonical) solutions.

We proceed by induction on $n$. The case $n\le2$ has already been proved. Assume, that the solution is unique for $n=m-1\ge2$. For $n=m$, equation~\eqref{eq:intertwining} with $i\in\{1,m\}$ implies that
\begin{eqnarray*}
(y_2,\ldots,y_{m-1})&=&\delta_1(g)(\partial_{g^{-1}(1)}(x_1,\ldots,x_{m-1})),\\
(y_1,\ldots,y_{m-2})&=&\delta_m(g)(\partial_{g^{-1}(m)}(x_1,\ldots,x_{m-1})),
\end{eqnarray*}
where $(y_1,\ldots,y_{m-1})=g(x_1,\ldots,x_{m-1})$. By the induction hypothesis, the right hand sides of these equations are uniquely defined, and so are their left hand sides. The latter, in turn, uniquely determine the $(m-1)$-tuple $(y_1,\ldots,y_{m-1})$, and, thus, the solution is unique for $n=m$.
\end{proof}
\begin{proposition}
The sub-group $A\subset B$, generated by the set of elements
\[
\cup_{n\ge1}\{ x+s_ix\vert\ x\in V_n,\  1\le i\le n\},
\] is a chain sub-complex.
\end{proposition}
\begin{proof} From the definition of $A$ it follows that $A=\oplus_{n\in\mathbb{Z}}A_n$ with $A_n=A\cap B_n$. Let us show that $\partial A_n\subset A_{n-1}$. For any $x\in V_n$ and $1\le i\le n$, we have
\begin{multline*}
\partial(x+s_ix)=\sum_{j=0}^{i-2}(-1)^j\partial_j(x+s_ix)+\sum_{j=i-1}^{i}(-1)^j\partial_j(x+s_ix)+
\sum_{j=i+1}^{n}(-1)^j\partial_j(x+s_ix)\\
=\sum_{j=0}^{i-2}(-1)^j(1+s_{i-1})\partial_jx+\sum_{j=i-1}^{i}(-1)^j(\partial_jx+\partial_{s_i(j)}x)+
\sum_{j=i+1}^{n}(-1)^j(1+s_i)\partial_jx\\
=\sum_{j=0}^{i-2}(-1)^j(1+s_{i-1})\partial_jx+
\sum_{j=i+1}^{n}(-1)^j(1+s_i)\partial_jx\in A_{n-1}.
\end{multline*}

\end{proof}
Thus, there is a short exact  sequence of chain complexes:
\[
0\to A\to B\to C=B/A\to 0.
\]
\begin{definition}
\emph{Integral homology} $H_*(G)$ of a $\Delta$-groupoid $G$ is the homology of the chain complex $C=A/B$.
\end{definition}

\begin{conjecture}\label{homol-conj}
Let $T$ and $T'$ be two ideal triangulations of a (cusped) 3-manifold, related by a $2-3$ Pachner move. Let $G$ and $G'$ be the associated $\Delta$-groupoids. Then, the corresponding homology groups $H_*(G)$ and $H_*(G')$ are isomorphic.
\end{conjecture}
\begin{remark}
Conjecture~\ref{homol-conj} is true for any  hyperbolic knot $K$  when the ideal triangulations are geometric. In this case the corresponding $\Delta$-groupoid homology is not interesting: one has the isomorphism $H_*(G_K)\simeq \tilde{H}_*(\mathbb{S}^3/K,\mathbb{Z})$, and it is well known fact that the homology of the space
$\mathbb{S}^3/K$ is independent of $K$:
 \[
\tilde{H}_n(\mathbb{S}^3/K,\mathbb{Z})=\left\{\begin{array}{cl}
\mathbb{Z},&\mathrm{if}\ n\in\{2,3\};\\
0,&\mathrm{otherwise}.
\end{array}
\right.
\]
 However, it might be interesting for non-hyperbolic knots, for example, one has the following results of calculations for the two simplest torus knots of types $(2,3)$ and $(2,5)$ (with some particular choices of ideal triangulations):
\[
H_n(G_{3_1})=\left\{\begin{array}{cl}
\mathbb{Z},&\mathrm{if}\ n=2;\\
0,&\mathrm{otherwise}.
\end{array}
\right.\quad
H_n(G_{5_1})=\left\{\begin{array}{cl}
\mathbb{Z},&\mathrm{if}\ n=2;\\
\mathbb{Z}_5,&\mathrm{if}\ n=3;\\
0,&\mathrm{otherwise}.
\end{array}
\right.
\]
\end{remark}
\section{Conclusion}\label{sec:concl}
The algebraic structure of $\Delta$-groupoids seems to be profoundly related with combinatorics of (ideal) triangulations of three-manifolds. Functorial relations of $\Delta$-groupoids to the category of rings permit us to construct ring-valued invariants which seem to be interesting. In the case of knots, these rings are universal for restricted class of representations of knot groups into the group $GL(2,R)$, where $R$ is an arbitrary ring.

Ideal triangulations of link complements give rise to presentations of associated $\Delta$-groupoids which, as groupoids with forgotten $\Delta$-structure, have as many connected components as the number of link components. In particular, they are connected groupoids in the case of knots. In general, two $\Delta$-groupoids associated with two ideal triangulations of one and the same knot complement are not isomorphic, but one can argue that the corresponding vertex groups are isomorphic. In this way, we come to the most evident $\Delta$-groupoid knot invariant to be called the \emph{vertex group} of a knot. It is not very sensitive invariant as one can show that it is trivial one-element group in the case of the unknot, isomorphic to the group of integers $\mathbb{Z}$ for any non-trivial torus knot, and isomorphic to the group $\mathbb{Z}^2=\mathbb{Z}\oplus \mathbb{Z}$ for any hyperbolic knot. Moreover, it is trivial at least for some connected sums, e.g. $3_1\# 3_1$ or $3_1\# 4_1$. In the light of these observations it would be interesting to calculate the vertex group for satellite knots which are not connected sums.

One can also refine the vertex group by adding extra information associated with a distinguished choice of a meridian-longitude pair. In this way one can detect the chirality of knots, for example, in the case of the torus knot $T_{p,q}$ of type $(p,q)$ and its mirror image $T_{p,q}^*$. In both cases, the vertex group, being the infinite cyclic group, is generated by the meridian $m$, while the longitude $l$ is given as $l=m^{pq}$ for $T_{p,q}$, and $l=m^{-pq}$ for $T_{p,q}^*$.

Finally, we have defined the integral $\Delta$-groupoid homology which seems not to be very interesting in the case of hyperbolic knots, but could be of some interest in the case of non-hyperbolic knots.
\subsection*{Acknowledgments} I would like to thank J\o rgen Ellegaard Andersen, Christian Blanchet, David Cimasoni, Francesco Costantino, Stavros Garoufalidis, Rostislav Grigorchuk, Pierre de la Harpe, Anatol Kirillov, Gregor Masbaum, Ashot Minasyan, Nikolai Reshetikhin, Leon Takhtajan, Vladimir Turaev, Oleg Viro, Roland van der Veen, Claude Weber for valuable discussions, and Michael Atiyah for his encouragement in this work.


\begin{thebibliography}{9}
\bibitem{Hatcher} Allen Hatcher, "Algebraic Topology". Cambridge University Press, Cambridge, 2002.
\bibitem{K} R. M.~Kashaev, \emph{On ring-valued invariants of topological pairs}, preprint arXiv:math/0701543.
\bibitem{KR}R. M.~Kashaev, N.~Reshetikhin, \emph{Symmetrically Factorizable Groups and Set-Theoretical Solutions of the Pentagon Equation}, Contemporary Mathematics, {\bf 433} (2007) 267--279.  See also math.QA/0111171.
\bibitem{ML}  S. Mac Lane, "Categories for the Working Mathematician".  Berlin: Springer-Verlag, 1971.
\end{thebibliography}
\end{document}